\documentclass{amsart}

\usepackage{amsmath, amssymb, amsthm, thmtools, amscd, amssymb, color, latexsym, tikz, hyperref, enumitem, float}
\usepackage[margin=1in]{geometry}
\usepackage[capitalize,noabbrev]{cleveref}
\usetikzlibrary{calc, external, fit, shapes, arrows.meta}

\numberwithin{equation}{section}
\newtheorem{theorem}{Theorem}[section]

\newtheorem{lemma}[theorem]{Lemma}
\newtheorem{question}[theorem]{Question}

\newtheorem{conjecture}[theorem]{Conjecture}
\crefname{conjecture}{Conjecture}{Conjectures}

\newtheorem*{theorem*}{Theorem}
\newtheorem*{proposition*}{Proposition}
\newtheorem*{corollary*}{Corollary}
\newtheorem*{lemma*}{Lemma}
\newtheorem*{question*}{Question}
\newtheorem*{claim*}{Claim}
\newtheorem*{conjecture*}{Conjecture}

\theoremstyle{definition}

\newtheorem{example}[theorem]{Example}
\crefname{example}{Example}{Examples}
\newtheorem{remark}[theorem]{Remark}

\newtheorem*{definition*}{Definition}
\newtheorem*{example*}{Example}
\newtheorem*{remark*}{Remark}
\newtheorem*{problem*}{Problem}
\newtheorem*{assumption*}{Assumption}

\usepackage{todonotes}

\def\E{\mathop{\mathbb{E}}}

\def\N{\mathbb{N}}

\def\P{\mathbb{P}}

\def\R{\mathbb{R}}

\def\Z{\mathbb{Z}}
\def\eps{\varepsilon}

\def\Base{\operatorname{base}}

\def\oto{\leftrightarrow}

\newcommand{\floor}[1]{\lfloor#1\rfloor}
\newcommand{\mc}[1]{\mathcal{#1}}

\newcommand{\abs}[1]{\left\lvert#1\right\rvert}
\newcommand{\set}[1]{\left\{#1\right\}}
\newcommand{\paren}[1]{\left(#1\right)}
\newcommand{\bracket}[1]{\left[#1\right]}
\newcommand{\1}[1]{1_{\left\{#1\right\}}}
\def\bra{\left<}
\def\ket{\right>}
\def\di{\,\operatorname{d}\!}
\def\Reff{\mathcal{R}_{\mathrm{eff}}}

\newcommand{\SortNoop}[1]{}

\usepackage{doi}

\begin{document}

\title[Recurrence and transience of random walks in monotonically changing environments]{Recurrence and transience of random walks in \\ monotonically changing environments}

\author[Rupert Li]{Rupert Li}
\address[]{Department of Mathematics, Stanford University, Stanford, CA 94305, USA}
\email{rupertli@stanford.edu}
\author[Jiyun Park]{Jiyun Park}
\address[]{Department of Mathematics, Stanford University, Stanford, CA 94305, USA}
\email{jiyunp@stanford.edu}

\begin{abstract}
Let $(c_t)_{t\geq 0}$ be a deterministic family of edge conductances on a countable vertex set, monotone in $t$, and let $(X_t)$ be the random walk that takes its $t$-th step using the conductances $c_t$.
We prove that if $c_t\uparrow c_\infty$ and $c_\infty$ is recurrent (respectively, $c_0$ is transient), then $(X_t)$ is almost surely recurrent (respectively, transient), i.e., visits every vertex infinitely (respectively, finitely) often.
We also establish the analogous results in continuous time.
This proves conjectures of Amir, Benjamini, Gurel-Gurevich, and Kozma, and the special case of $\set{0,1}$-valued conductances corresponds to simple random walk on a growing graph, and in this special case our results prove a conjecture of Dembo, Huang, and Sidoravicius.
In addition, we provide counterexamples to the corresponding conjectures when $(c_t)$ is monotone non-increasing: if $c_t\downarrow c_\infty$ and $c_\infty$ is transient, $(X_t)$ need not be transient, and similarly if $c_0$ is recurrent, $(X_t)$ need not be recurrent, even if $c_\infty\geq\alpha c_0$ for some $\alpha>0$. 
\end{abstract}

\maketitle

\section{Introduction}\label{sec:intro}
Consider a random walk whose environment changes as it walks: at time $t$ the walker traverses an incident edge with probability proportional to its conductance, with the edge conductances at time $t$ given by the vector $c_t$, indexed over pairs of vertices (where if there is no conductance incident to the current vertex, no move is made).
When the environment is constant in time, Rayleigh monotonicity~\cite{Rayleigh} implies that recurrence is monotone in the conductances, i.e, if two sets of edge conductances $c$ and $d$ satisfy $c(e)\leq d(e)$ for all edges $e$, then random walk corresponding to $d$ being recurrent implies random walk corresponding to $c$ is recurrent; see~\cite{LyonsPeres2016}.
Whether any form of this monotonicity survives for time-varying environments is a well-known open problem, raised independently by Dembo, Huang, and Sidoravicius~\cite{DHS2014} for growing subgraphs, i.e., the case where the $c_t$ are $\set{0,1}$-valued, and by Amir, Benjamini, Gurel-Gurevich, and Kozma~\cite{ABGK2020} for general monotone conductances.
Amir, Benjamini, Gurel-Gurevich, and Kozma~\cite[Conjecture 7.1]{ABGK2020} state the following four conjectures regarding monotonicity of recurrence and transience in monotone increasing or decreasing environments.
\begin{conjecture}[cf.\ \protect{\cite[Conjecture 7.1]{ABGK2020}}]\label{conj:inc_rec}
    Suppose random walk $(X_t)_{t\geq0}$ in changing environment corresponding to deterministic conductances $(c_t)_{t\geq0}$, i.e., at time $t$ the walker traverses an edge according to edge conductances $c_t$, has $(c_t)_{t\geq0}$ non-decreasing with limit $c_\infty$, i.e., $c_t\uparrow c_\infty$.
    If random walk corresponding to $c_\infty$ is irreducible and recurrent, then $(X_t)_{t\geq0}$ is recurrent, in the sense that almost surely the walk visits every vertex infinitely often.
\end{conjecture}
Amir, Benjamini, Gurel-Gurevich, and Kozma~\cite{ABGK2020} originally stated their conjecture under the assumption that $c_t$ is non-decreasing and bounded above by $c_\infty$, but this is equivalent (upon properly defining recurrence if $c_\infty$ is supported on a superset of the vertices that $(X_t)_{t\geq0}$ can reach) to our above statement via Rayleigh monotonicity.
Henceforth, for simplicity, we will often assume the bounding networks, e.g., $c_0$ or $c_\infty$, are irreducible; our results can all be immediately extended to remove this assumption, upon considering recurrence or transience of specific vertices.

In addition, much of the work of~\cite{ABGK2020} focuses on the \emph{adaptive} setting, where $c_t$ can depend on the motion of the random walk, i.e., on $(X_s:s\leq t)$; their statement of the above conjecture for the \emph{non-adaptive} setting allows for $(c_t)_{t\geq0}$ to be a random process, just independent of the motion of the random walk.
However, the random non-adaptive case immediately follows from the deterministic case that we stated, by simply conditioning on the process $(c_t)_{t\geq0}$; henceforth, for simplicity, we will only discuss deterministic environments $(c_t)_{t\geq0}$, but note that our results immediately generalize to the full non-adaptive setting. 
\begin{conjecture}[cf.\ \protect{\cite[Conjecture 7.1]{ABGK2020}}]\label{conj:inc_trans}
    Suppose random walk $(X_t)_{t\geq0}$ in changing environment corresponding to deterministic conductances $(c_t)_{t\geq0}$ has $(c_t)_{t\geq0}$ non-decreasing with limit $c_\infty$, i.e., $c_t\uparrow c_\infty$.
    If random walk corresponding to $c_0$ is irreducible and transient, then $(X_t)_{t\geq0}$ is transient, in the sense that almost surely the walk visits every vertex finitely often.
\end{conjecture}
The existence of the limit $c_\infty$ is clearly needed, as observed by \cite[Remark 1.9]{DHS2014}; otherwise, a simple example, such as the one discussed in \cref{ex:inc_trans_unbounded}, is to rapidly increase $c_t$ on a recurrent subset of the network, e.g., $\Z^2\subseteq\Z^3$, so that $(X_t)_{t\geq0}$ essentially behaves as a random walk on a fixed recurrent network.
\begin{conjecture}[cf.\ \protect{\cite[Conjecture 7.1]{ABGK2020}}]\label{conj:dec_trans}
    Suppose random walk $(X_t)_{t\geq0}$ in changing environment corresponding to deterministic conductances $(c_t)_{t\geq0}$ has $(c_t)_{t\geq0}$ non-increasing with limit $c_\infty$, i.e., $c_t\downarrow c_\infty$.
    If random walk corresponding to $c_\infty$ is irreducible and transient, then $(X_t)_{t\geq0}$ is transient, in the sense that almost surely the walk visits every vertex finitely often.
\end{conjecture}
\begin{conjecture}[cf.\ \protect{\cite[Conjecture 7.1]{ABGK2020}}]\label{conj:dec_rec}
    Suppose random walk $(X_t)_{t\geq0}$ in changing environment corresponding to deterministic conductances $(c_t)_{t\geq0}$ has $(c_t)_{t\geq0}$ non-increasing with limit $c_\infty$, i.e., $c_t\downarrow c_\infty$.
    If $c_\infty \geq \alpha c_0$ for some $\alpha>0$ and random walk corresponding to $c_0$ is irreducible and recurrent, then $(X_t)_{t\geq0}$ is recurrent, in the sense that almost surely the walk visits every vertex infinitely often.
\end{conjecture}
The constant-factor assumption, i.e., the existence of $\alpha>0$, is needed; \cite[Example 4.5]{ABGK2020} observed that the simple example of random walk on $\N$ with
\[ c_t(j,j+1) = \begin{cases} 2^{-j} & \text{if } j < t \\ 1 & \text{otherwise} \end{cases} \]
starts with $c_0\equiv1$, where simple random walk on $\N$ is recurrent, but has positive probability of always going to the right.
With this example, we remark that a 0-1 law does not hold for the more general setting of random walk in changing environment, unlike the classical theory for time-homogeneous Markov chains: this example is of \emph{mixed type}, as with positive probability it always goes to the right, i.e., $X_t=t$ for all $t\geq0$, and so is transient, but if it ever goes to the left, which also occurs with positive probability, then $(X_t)_{t\geq0}$ henceforth behaves as random walk on the fixed network $c(j,j+1)=2^{-j}$, which moves left at each positive state with probability $\frac23$, and so is recurrent.

\cref{conj:inc_rec} is a generalization of the following conjecture of Dembo, Huang, and Sidoravicius~\cite{DHS2014}.
\begin{conjecture}[\protect{\cite[Conjecture 1.10]{DHS2014}}]\label{conj:dhs}
If simple random walk on a fixed graph $G_\infty$ of uniformly bounded degrees is recurrent, then the same applies to simple random walk on $(G_t)_{t\geq0}$ starting at the same point, for any choice of $G_t\uparrow G_\infty$, in the sense that almost surely the walk returns to its starting point infinitely often.
\end{conjecture}

\Cref{conj:dhs} arose from the study in~\cite{DHS2014} of random walks on growing domains $D_t\uparrow\R^d$, where the recurrence-versus-transience dichotomy was established for $d\geq3$ under regularity assumptions on the growth (e.g., requiring the domain $D_t$ be bounded on both sides by Euclidean balls with radii tending to infinity); see also Dembo, Huang, Morris, and Peres~\cite{DHMP2017} for transience criteria for growing subgraphs via evolving sets. Amir, Benjamini, Gurel-Gurevich, and Kozma~\cite{ABGK2020} proved \cref{conj:inc_rec,conj:inc_trans,conj:dec_trans,conj:dec_rec} on $\N$, where it holds even for adaptive environments, and moreover also proved \cref{conj:inc_rec,conj:dec_trans} more generally for trees, again even for adaptive environments. Park and Ray~\cite{ParkRay} prove recurrence and transience for ``slowly changing environments'', i.e., graphs (which are allowed to be adaptive and non-monotone) with bounds on how much they are allowed to change over time. The general conjectures, with perhaps the quintessential example being random walk on growing subgraphs of $\Z^d$, were all open.

Beyond recurrence and transience, understanding various properties of time-inhomogeneous Markov chains (both monotone and otherwise) is a topic of widespread interest. Major areas of interest include cover times and mixing times~\cite{AGHHHN2022, AKM2018, SauerwaldZanetti, ShimizuShiraga}, Gaussian heat kernel estimates~\cite{DHZ2019, HuangKumagai}, and merging~\cite{Moumeni,SCZ2011a, SCZ2011b}. See also Huang~\cite{Huang} for transition probability bounds, and Dolgopyat, Keller, and Liverani~\cite{DKL2008}, and Dolgopyat and Liverani~\cite{DolgopyatLiverani} for central limit theorems.

Our main results prove the two monotone non-decreasing questions, \cref{conj:inc_rec,conj:inc_trans}, and disprove the two monotone non-increasing questions, \cref{conj:dec_trans,conj:dec_rec}.
To standardize our notation, a \emph{network} $(V,c)$ is a countable vertex set $V$ with symmetric conductances $c: V\times V\to[0,\infty)$; its edges are the pairs with $c(e)>0$.
The \emph{total conductance} of a vertex is $\pi(y)=\sum_z c(y,z)$, which we henceforth assume is always finite at every vertex so that the random walk is well-defined.
If the random walk is at $y$, it jumps to $z$ with probability $c(y,z)/\pi(y)$.
We emphasize that no local finiteness is assumed: a vertex may be incident to infinitely many edges, as long as its total conductance is finite, which is necessary for the walk to be defined.
\begin{theorem}\label{thm:inc_rec}
    \cref{conj:inc_rec} and thus \cref{conj:dhs} are true.
\end{theorem}
\begin{remark}\label{rmk:assumptions}
    Our restriction to non-adaptive conductances cannot be lifted; see \cite[Theorem 6.1]{ABGK2020} for an example of a monotone non-decreasing adaptive random walk on $\Z^2$, with conductances in $[1,2]$ (so bounded between two recurrent networks), whose walk is transient.
    And our assumption that $(c_t)_{t\geq0}$ is monotone also cannot be removed: consider the simple example on $\N$ where $c_t(i,i+1)=1-\1{t=i+1}$, so that $c_t\to c_\infty$ where $c_\infty$ is identically 1, which corresponds to simple random walk on $\N$, which is recurrent.
    However, $X_t=t$ almost surely, and thus is transient.
    Thus, in some sense \cref{thm:inc_rec} is the most general result possible of this form.
\end{remark}
\begin{theorem}\label{thm:inc_trans}
    \cref{conj:inc_trans} is true. The expected number of visits to any point is also finite and bounded by
    \[  
    \sum_{t = 0}^{\infty} \mathbb{P}(X_t = x) \le 2\pi_{\infty}(x) \Reff^{G_0}(x \oto \infty).
    \]
    Furthermore, the escape probability is strictly positive, i.e.,
    \[
    \mathbb{P}(X_t \ne x \text{ for all }t > 0 | X_0 = x) > 0.
    \]
\end{theorem}
This result also resolves an open problem of Dembo, Huang, Morris, and Peres~\cite[Problem 1.17(a)]{DHMP2017}, which asks for this result in the special case of graphs, where all $G_t$ are subgraphs of $\Z^d$ for $d\geq3$, and $G_0$ is the unique infinite cluster of super-critical Bernoulli bond percolation; the result immediately follows as a special case of \cref{thm:inc_trans}, as Grimmett, Kesten, and Zhang~\cite{GKZ1993} showed that the infinite cluster is almost surely transient.
\begin{theorem}\label{thm:dec_trans}
    \cref{conj:dec_trans} is false: there exist a connected network $(V,c_\infty)$ with $\pi_\infty(v)<\infty$ for all $v$, a vertex $v_0\in V$, and a deterministic non-increasing environment $c_t\downarrow c_\infty$ with $\pi_t(v)<\infty$, such that every network $c_t$ for $t\in\N\cup\set{\infty}$ is transient, and yet the walk $(X_t)_{t\ge0}$ with $X_0=v_0$ almost surely visits every vertex of $V$ infinitely often.
\end{theorem}
\begin{theorem}\label{thm:dec_rec}
    \cref{conj:dec_rec} is false: for every $\delta>0$, there exists a connected network $(V,c_0)$ with $\pi_0(v)<\infty$ for all $v$, a vertex $v_0\in V$, $\alpha\in(0,1)$, and a deterministic non-increasing environment $c_t\downarrow c_\infty\ge\alpha c_0$, in which every edge changes at most once, such that every network $c_t$ for $t\in\N\cup\set{\infty}$ is recurrent, and yet the walk $(X_t)_{t\ge0}$ with $X_0=v_0$ visits every vertex of $V$ finitely often with probability at least $1-\delta$.
\end{theorem}
In some sense, at an intuitive level, these counterexamples exist for the monotone non-increasing setting because the non-increasing setting allows structures to be removed, where under certain timings of this removal, the resulting random walk in changing environment can have different behavior than its limiting network, which does not see these temporary structures.
For example, our construction for \cref{thm:dec_trans} consists of a random walk on $\Z$ with positive drift, but with long edges of high conductance between $n+1$ and $-n$, which are slowly and sequentially removed, such that each edge sends the random walk backwards, thus ensuring it visits 0 infinitely often.
Similarly, our construction for \cref{thm:dec_rec} essentially consists of a driftless random walk on $\N$, which is recurrent, but the conductances are decreased slowly from left to right, so that, following this ``wavefront,'' the random walk experiences positive drift and thus is transient; the limiting network has all conductances decreased by the same factor, so is again driftless and thus recurrent.
In the monotone non-decreasing setting of \cref{thm:inc_rec,thm:inc_trans}, edges are permanently added, rather than removed, and so the limiting graph still possesses these structures; similarly, any ``wave'' example would pull the random walk in the opposite direction as the direction of propagation, unlike the non-increasing setting where the wave can push the random walk along the direction of propagation, and thus the wave can only have a temporary effect on the random walk in the non-decreasing setting, where the monotonicity results do hold.

Amir, Benjamini, Gurel-Gurevich, and Kozma \cite[Question 7.3]{ABGK2020} also ask whether analogous results hold for continuous-time random walks in changing environments.
In this setting, with conductances $(c_t)_{t\geq0}$ now for all real $t\geq0$, each edge $e$ is equipped with a Poisson clock with rate $c_t(e)$ at time $t$, for all $t\geq0$.
When a clock rings, if the random walk is adjacent to that edge at that time, then the random walk traverses that edge.
We prove the continuous time analogs of our positive results, \cref{thm:inc_rec,thm:inc_trans}.
Recall that for continuous-time random walks, letting $T_n$ denote the time of the $n$-th jump, it is possible for $\lim_{n\to\infty} T_n < \infty$; then $\zeta=\lim_{n\to\infty} T_n$ is referred to as the \emph{explosion time}.
We refer readers to~\cite{Norris} for standard background on Markov chains, including continuous-time Markov chains.
\begin{theorem}\label{thm:ct_inc_rec}
    Suppose continuous-time random walk $(X_t)_{t\geq0}$ in changing environment corresponding to deterministic conductances $(c_t)_{t\geq0}$, has $(c_t)_{t\geq0}$ non-decreasing with limit $c_\infty$, i.e., $c_t\uparrow c_\infty$.
    If random walk corresponding to $c_\infty$ is irreducible and recurrent, then $(X_t)_{t\geq0}$ does not explode, i.e., it makes finitely many jumps in every bounded time interval, and is recurrent, in the sense that $\set{t:X_t=v}$ is unbounded for all $v$.
\end{theorem}
\begin{theorem}\label{thm:ct_inc_trans}
    Suppose continuous-time random walk $(X_t)_{t\geq0}$ in changing environment corresponding to deterministic conductances $(c_t)_{t\geq0}$ has $(c_t)_{t\geq0}$ non-decreasing with limit $c_\infty$, i.e., $c_t\uparrow c_\infty$.
    If random walk corresponding to $c_0$ is irreducible and transient, then $(X_t)_{t\geq0}$ is transient, in the sense that, for explosion time $\zeta\in[0,\infty]$, the expected number of visits to any vertex is finite and the expected local time $\E\bracket{\int_0^\zeta\1{X_t=v}\di t}$ is finite for all $v$.
\end{theorem}
\begin{remark}
    While we believe this definition of a continuous-time random walk in changing environment $(c_t)_{t\geq0}$ is the most natural, and is the one considered by Amir, Benjamini, Gurel-Gurevich, and Kozma~\cite{ABGK2020}, an alternative definition that has been considered, e.g., by Dembo, Huang, Morris, and Peres~\cite{DHMP2017}, is the \emph{constant-speed} random walk, which moves whenever a Poisson clock with rate 1 rings, and upon ringing at time $t$ selects a move according to the conductances $c_t$.
    In this setting, the explosion time $\zeta=\infty$ almost surely, and the analogs of \cref{thm:ct_inc_rec,thm:ct_inc_trans} immediately follow from \cref{thm:inc_rec,thm:inc_trans} by simply conditioning on the times $T_1,T_2,\dots$ that the clock rings and applying the results to $(c_{T_n})_{n\geq0}$.
\end{remark}
\subsection{Section overview}
In \cref{sec:prelim} we review classical background of potential theory and electrical networks, introduce our notation, and prove some basic observations for our monotone non-decreasing setting.
In \cref{sec:inc} we prove the positive results for monotone non-decreasing random walks, \cref{thm:inc_rec,thm:inc_trans}.
In \cref{sec:ct} we prove the continuous time analogs of our positive results, i.e., \cref{thm:ct_inc_rec,thm:ct_inc_trans}.
In \cref{sec:dec_trans,sec:dec_rec} we prove the counterexamples for monotone non-increasing random walks, i.e., \cref{thm:dec_trans,thm:dec_rec}, respectively.
Finally, in \cref{sec:open} we record some open questions.

\section{Preliminaries}\label{sec:prelim}
\subsection{Review of potential theory and electrical networks} \label{subsec:review}
We review classical theory of random walks on electrical networks, and refer readers to~\cite{LyonsPeres2016} for an overview of the theory.
Let $G = (V, E, c)$ be a (potentially infinite) multigraph (i.e., it may have loops and multiple edges between the same pair of vertices) with edge conductances $c : E \to (0, \infty)$.
$E$ is the (directed) edge set of $G$, and $c(e)$ is the conductance of edge $e \in E$ with reciprocal $r(e) = 1/c(e)$.
Each edge $e \in E$ has a tail vertex $e^-$ and a head vertex $e^+$.
We assume $G$ is undirected, so that every edge $e \in E$ has a reverse edge $-e \in E$ with $(-e)^- = e^+$, $(-e)^+ = e^-$, and $c(-e) = c(e)$.
We assume $G$ has finite total conductance, i.e., $\sum_{e^- = v} c(e) < \infty$ for all $v \in V$; this is needed to define the random walk on $G$.
It may still be that $v$ has infinitely many neighbors.

Let $\ell^2(V)$ be the Hilbert space of functions $f : V \to \mathbb{R}$ with inner product 
\[
\langle f, g \rangle = \sum_{v \in V} f(v) g(v).
\]
Let $\ell^2_{-} (E)$ be the space of antisymmetric functions $j : E \to \mathbb{R}$, i.e., satisfying $j(-e) = -j(e)$ for all $e \in E$, with inner product
\[
\langle j, k \rangle = \frac{1}{2} \sum_{e \in E} j(e) k(e).
\]
The factor of $\frac12$ is included to avoid double counting edges, or equivalently, one can take $\langle j, k \rangle = \sum_{e \in E_{1/2}} j(e) k(e)$ where $E_{1/2}$ is a choice of orientation of the edges of $G$.

Define the boundary operator $\di^* : \ell^2_{-}(E) \to \ell^2(V)$ and coboundary operator $\di : \ell^2(V) \to \ell^2_{-}(E)$ by
\[
(\di^* j)(v) = \sum_{e^- = v} j(e), \qquad (\di f)(e) = f(e^-) - f(e^+).
\]
Note that they satisfy the relation $\langle \di^* j, f \rangle = \langle j, \di f \rangle$ for all $j \in \ell^2_{-}(E)$ and $f \in \ell^2(V)$. The gradient operator $\nabla : \ell^2(V) \to \ell^2_{-}(E)$ is defined by
\[
(\nabla f)(e) = c(e) (f(e^-) - f(e^+)) = c(e) (\di f)(e).
\]
The Laplacian operator $\Delta : \ell^2(V) \to \ell^2(V)$ is defined as $\Delta = \di^* \nabla$, or equivalently,
\[
(\Delta f)(v) = \sum_{e^- = v} c(e) (f(v) - f(e^+)) = \sum_{e^- = v} c(e) (\di f)(e).
\]
We define the Dirichlet energy of a function $f \in \ell^2(V)$ by
\[
\mathcal{E}(f) = \langle f, \Delta f \rangle = \langle \nabla f, \nabla f \rangle_r = \langle \di f, \di f \rangle_c,
\]
where $\langle j, k \rangle_h = \langle jh, k \rangle$ is the weighted inner product on $\ell^2_{-}(E)$ with weight $h : E \to (0, \infty)$. More generally, the Dirichlet form $\mathcal{E} : \ell^2(V) \times \ell^2(V) \to \mathbb{R}$ is defined by
\[
\mathcal{E}(f, g) = \langle f, \Delta g \rangle = \langle \nabla f, \nabla g \rangle_r = \langle \di f, \di g \rangle_c
\]
and the energy of a flow $j \in \ell^2_{-}(E)$ is defined by
\[
\| j \|_r^2 = \langle j, j \rangle_r = \frac{1}{2} \sum_{e \in E} r(e) j(e)^2.
\]
As a consequence, we may deduce that $\mathcal{E}(f) = \| \nabla f \|_r^2$.

The \emph{effective resistance} $\Reff^G(A \leftrightarrow B)$ between two disjoint sets $A$, $B$ satisfies Thomson's principle:
\[
\Reff^G(A \leftrightarrow B) = \inf \left\{ \| j \|_r^2 : j \in \ell^2_-(E), \di^* j |_{(A \cup B)^c} = 0 , \sum_{e^- \in A} j(e) = 1 \right\}.
\]
When $(A \cup B)^c$ is finite, equality is obtained by the \emph{unit current flow} $i_{A, B}$ which satisfies $i_{A,B} = \nabla \psi_{A, B}$ for some voltage function $\psi_{A, B}$. This voltage function is harmonic on $(A \cup B)^c$ and may be chosen such that $\psi_{A, B} = \Reff(A \leftrightarrow B)$ on $A$ and $0$ on $B$. Rescaling by $\Reff^G(A \leftrightarrow B)$, this also leads to Dirichlet's principle:
\[  
    \Reff^G(A \leftrightarrow B)^{-1} = \inf \left\{ \| \nabla f \|_r^2 : f|_A = 1 , f|_B = 0 \right\}.
\]
Equality is obtained by the harmonic measure $\phi_{A, B}$ with boundary conditions $\phi_{A, B} = 1$ on $A$ and $\phi_{A, B} = 0$ on $B$.
Rayleigh's monotonicity principle states that if $G_1 \subseteq G_2$, or equivalently if their edge conductances satisfy $c_1 \leq c_2$, then
\[  
\Reff^{G_1}(A \leftrightarrow B) \ge \Reff^{G_2}(A \leftrightarrow B).
\]

When $G$ is infinite, define
\[  
    \Reff^G(x \leftrightarrow \infty) = \lim_{B \uparrow V} \Reff^{G}(x \leftrightarrow B^c)
\]
where $B \uparrow V$ denotes a finite set growing to exhaust all of $V$.
It is well-known that the escape probability is given by
\[  
\mathbb{P}(X_t \ne x \text{ for all } t > 0 | X_0 = x) = \frac{1}{\pi(x) \Reff(x \oto \infty)}
\]
and therefore the random walk on $G$ starting at $x$ is recurrent if and only if $\Reff^G(x \leftrightarrow \infty) = \infty$. Hence for two graphs $G_1 \subseteq G_2$, Rayleigh's monotonicity principle implies that if $G_2$ is recurrent, then $G_1$ is also recurrent. Similarly, if $G_1$ is transient, then $G_2$ is also transient.

\subsection{Basic observations}\label{subsec:basics}
Here we explain the basic setup and lemmas that will be used in the proofs of \cref{thm:inc_rec,thm:inc_trans}.
Instead of increasing the conductance of an edge from $c_t(e)$ to $c_{t+1}(e)$, we shall always add a new parallel edge with conductance $c_{t+1}(e) - c_t (e)$; this is of course equivalent, but this convention will trivialize both the statement and proof of \cref{lemma:IBP}.
In this context, we define network inclusion $G_1 \subseteq G_2$ to mean that $G_1$ can be obtained by erasing edges from $G_2$.
Equivalently, we may view edge removal as changing their conductance to zero.
While our definition of networks only allowed for strictly positive conductances, it is easy to check that all of the definitions in Section~\ref{subsec:review} generalize easily.
The only subtlety is the inner product $\langle \cdot, \cdot \rangle_r$, where we have to be careful to only take such an inner product on edges with strictly positive conductance.

Now let $G_t = (V, E_t, c)$ be a monotone non-decreasing sequence of graphs, i.e., $E_t \subseteq E_{t+1}$ for all $t \ge 0$, where letting $E = \bigcup_{t \ge 0} E_t$, we have $c : E \to (0, \infty)$ a positive weight function. The limiting graph will be denoted by $G = (V, E, c)$, with resistance $r$. For brevity of notation, subscripts on the operators $\Delta$ and $\nabla$, e.g., $\Delta_t$ and $\nabla_t$, as well as the resistance $r$, will denote the operators and resistances corresponding to the graph $G_t$.
We view $\ell_-^2(E_s)\subseteq\ell_-^2(E_t)\subseteq\ell_-^2(E)$ for $s<t$ by extending functions by zeros.
Our most important observation is the following lemma.
\begin{lemma}\label{lemma:IBP}
    For any three networks $G_1 \subseteq G_2 \subseteq G_3$ and $f, g \in \ell^2(V)$,
    \[
    \langle \nabla_1 f, \nabla_2 g \rangle_{r_3} = \langle \Delta_1 f, g \rangle
    \]
\end{lemma}
\begin{proof}
    Note that $\nabla_1 f$ is supported on $E_1$. Since $\nabla_1 = \nabla_2$ and $r_1 = r_3$ on $E_1$, this implies
    \begin{align*}
        \langle \nabla_1 f, \nabla_2 g \rangle_{r_3} &= \langle \nabla_1 f , \nabla_1 g \rangle_{r_1} = \langle \Delta_1 f, g \rangle. \qedhere
    \end{align*}
\end{proof}
Let $X_t$ be a nearest-neighbor random walk on $G_t$, starting at a vertex $x\in V$.
Consider a stopping time $\tau$ and let $X_{t \wedge \tau}$ be the stopped process. We define two measures on $V$:
\begin{align*}
\mu_t(v) &= \mathbb{P}(X_{t \wedge \tau} = v), \\
\nu_t(v) &= \mathbb{P}(X_t = v, t < \tau).
\end{align*}
We say that $X_{t \wedge \tau}$ is \emph{alive} if $t < \tau$ and \emph{dead} otherwise. With this language, $\mu_t$ is the distribution of $X_{t \wedge \tau}$ and $\nu_t$ is the measure of the living $X_{t \wedge \tau}$. We also define
\[
\rho_t(v)  = \frac{1}{\pi_t(v)} \nu_t(v).
\]
The following observation lets us relate $\rho_t$ with $\mu_t$, and will be critical to our arguments.
\begin{lemma}\label{lemma:Laplacian_flow}
    For any $t \ge 0$,
    \begin{equation}
      \label{eq:laplacian}
    (\Delta_t \rho_t)(v) = \mu_t(v) - \mu_{t+1}(v).
    \end{equation}
    Therefore, for any $\phi \in \ell^2(V)$ and $0 \le s \le t$, we have
    \begin{equation}
      \label{eq:rec-eq}
    \langle \nabla_s \rho_s , \nabla_t \phi \rangle_r = \langle \mu_s - \mu_{s+1} , \phi \rangle
    \end{equation}
    and
    \begin{equation}
      \label{eq:trans-eq}
    \langle \nabla_t \rho_t , \nabla_s \phi \rangle_r = \langle \rho_t , \Delta_s \phi \rangle.
    \end{equation}
\end{lemma}
\begin{proof}
  First observe that
  \[
  \nabla_t \rho_t (e) = c(e) (\rho_t(e^-) - \rho_t(e^+)) = \frac{c(e)}{\pi_t(e^-)} \nu_t(e^-) - \frac{c(e)}{\pi_t(e^+)} \nu_t(e^+)
  \]
  equals the net flow of $X_{t \wedge \tau}$ along edge $e$ at time $t$. Therefore, summing over all multi-edges between $v$ and its neighbor $u$, we have
  \[
  \sum_{\substack{e^- = v \\ e^+ = u}} \nabla_t \rho_t(e) = \mathbb{P}(X_{t \wedge \tau} = v, X_{(t+1) \wedge \tau} = u) - \mathbb{P}(X_{t \wedge \tau} = u, X_{(t+1) \wedge \tau} = v).
  \]
  Summing over all neighbors $u$ of $v$, we obtain
  \[
  (\Delta_t \rho_t)(v) = \sum_{e^- = v} \nabla_t \rho_t(e) = \mathbb{P}(X_{t \wedge \tau} = v) - \mathbb{P}(X_{(t+1) \wedge \tau} = v) = \mu_t(v) - \mu_{t+1}(v)
  \]
  which proves \eqref{eq:laplacian}. Now by combining this with \cref{lemma:IBP}, we obtain \eqref{eq:rec-eq} and \eqref{eq:trans-eq}.
\end{proof}
We conclude this subsection with the following bound on $J_T = \sum_{t = 0}^{T-1} \nabla_t \rho_t$.
\begin{lemma}\label{lemma:J_bound}
  Let $J_T = \sum_{t=0}^{T-1} \nabla_t \rho_t$. Then
  \[
    \| J_T \|_r^2 \le 2 \sum_{t = 0}^{T-1} \rho_t(x).
  \]
\end{lemma}
\begin{proof}
  We may write
  \[
  \| J_T \|_r^2 = \sum_{t = 0}^{T-1} \| \nabla_t \rho_t \|_r^2 + 2\sum_{t = 1}^{T-1} \sum_{s = 0}^{t-1} \langle \nabla_s \rho_s , \nabla_t \rho_t \rangle_r.
  \]
  The off-diagonal terms telescope as
  \[
  \sum_{s = 0}^{t-1} \langle \nabla_s \rho_s , \nabla_t \rho_t \rangle_r = \sum_{s = 0}^{t-1} \langle \Delta_s \rho_s , \rho_t \rangle = \sum_{s = 0}^{t-1} \langle \mu_s - \mu_{s+1} , \rho_t \rangle = \langle \mu_0 - \mu_t , \rho_t \rangle
  \]
  by \cref{lemma:Laplacian_flow}, while the diagonal terms are similarly
  \[
  \langle \nabla_t \rho_t , \nabla_t \rho_t \rangle_r = \langle \Delta_t \rho_t , \rho_t \rangle = \langle \mu_t - \mu_{t+1} , \rho_t \rangle.
  \]
  Therefore,
  \begin{align*}
    \| J_T \|_r^2 &= \sum_{t = 0}^{T-1} \left( \langle \mu_t - \mu_{t+1} , \rho_t \rangle + 2\langle \mu_0 - \mu_t , \rho_t \rangle \right) \\
    &= \sum_{t = 0}^{T-1} \left( 2 \langle \mu_0 , \rho_t \rangle - \langle \mu_t , \rho_t \rangle - \langle \mu_{t+1} , \rho_t \rangle \right) \\
    &\le 2 \sum_{t = 0}^{T-1} \rho_t(x).
  \end{align*}
  as $\mu_0=\delta_x$.
\end{proof}
\section{Increasing recurrence and transience}\label{sec:inc}
We first prove \cref{thm:inc_rec}.
Let the random walk start at $X_0=x$.
Note that it suffices to prove that $X_t$ always returns to $x$ with probability $1$.
Indeed, if $X_t$ returns to $x$ at some time $t_0$, potentially a random stopping time, then by the strong Markov property the process $X_{t_0 + t}$ is a monotone random walk starting at $x$, and hence will also return to $x$ with probability $1$, and so $X_t$ almost surely visits $x$ infinitely often.
For any other point $y$ in the connected component of $x$ (in $G$), choose some path $\gamma$ from $x$ to $y$. This walk becomes available at some finite time, after which the probability of $X_t$ following $\gamma$ immediately after visiting $x$ is at least 
\[
\prod_{e \in \gamma} \frac{c(e)}{\pi(e^-)}
\]
Since $X_t$ visits $x$ infinitely often, $X_t$ will visit $y$ infinitely often as well. Thus it suffices to prove that $X_t$ returns to $x$ with probability $1$.
\begin{proof}[Proof of \cref{thm:inc_rec}]
  Let $B$ be a finite set containing $x$.
  Define the stopping time
  \[ \tau_B = \inf \{ t > 0 : X_t \in B^c \text{ or } X_t = x \} \]
  and define $\mu_t, \nu_t, \rho_t$ as in \cref{subsec:basics} with respect to $\tau_B$.
  Let $\phi_B$ be the harmonic function on $G$ with boundary conditions $\phi_B(x) = 0$ and $\phi_B|_{B^c} = 1$.
  We know that $\| \nabla \phi_B \|_r^2 = \Reff^{G}(x \leftrightarrow B^c)^{-1}$.
  Moreover, recalling the definition of $J_T$ from \cref{lemma:J_bound}, by \cref{lemma:IBP,lemma:Laplacian_flow},
  \[
    \langle J_T, \nabla \phi_B \rangle_r = \sum_{t = 0}^{T-1} \langle \nabla_t \rho_t , \nabla \phi_B \rangle_r = \sum_{t = 0}^{T-1} \langle \Delta_t \rho_t , \phi_B \rangle = \sum_{t = 0}^{T-1} \langle \mu_t - \mu_{t+1} , \phi_B \rangle = \langle \mu_0 - \mu_T , \phi_B \rangle = - \mathbb{E}[\phi_B(X_{T \wedge \tau_B})].
  \]
  By Cauchy--Schwarz and \cref{lemma:J_bound}, we have
  \[
    \mathbb{E}[\phi_B(X_{T \wedge \tau_B})] = -\langle J_T, \nabla \phi_B \rangle_r \le \| J_T \|_r \| \nabla \phi_B \|_r \le \sqrt{\frac{2}{\pi_0(x) \Reff^{G}(x \leftrightarrow B^c)}}.
  \]
  The last inequality comes from the fact that $\rho_t(x) = 0$ for $t > 0$ and $\rho_0(x)=\frac{1}{\pi_0(x)}$.
  As $\tau_B<\infty$ almost surely since $B$ is finite, taking the limit as $T\to\infty$ and then the limit as $B\uparrow V$ so that $\Reff^G(x\oto B^c)\to\infty$, we obtain
  \[
  \lim_{B \uparrow V} \mathbb{E}[\phi_B(X_{\tau_B})] = \lim_{B\uparrow V} \P(X_{\tau_B}\in B^c) = 0,
  \]
  so $X_t$ returns to $x$ almost surely.
\end{proof}
The proof of \cref{thm:inc_trans} follows the same lines as the proof of \cref{thm:inc_rec} but for two changes: we do not kill the walk, and the accumulated flow $J_T$ is tested against the harmonic potential on $G_0$ instead of $G$.
Thus, instead of bounding the escape probability $\P(X_{\tau_B}\in B^c)$, we bound the expected number of visits to $x$ by the walk $X_t$ in terms of the effective resistance $\Reff^{G_0}(x\oto\infty)$.
This only shows that the expected number of visits is finite; to prove that the return probability is strictly less than 1, i.e., the escape probability is strictly positive, it suffices to show that for sufficiently large $N$, replacing the initial graphs $G_0,G_1,\dots,G_N$ all with $G_0$ (which yields an absolutely continuous distribution compared to the original law of $(X_t)_{t\geq0}$), we escape with positive probability.
Not returning before time $N$ is bounded below by the escape probability on $G_0$, which is strictly positive, and the probability of returning after time $N$ but not before time $N$ tends to 0 as $N\to\infty$, via a modification of our previous argument bounding the expected number of visits to $x$, stopped if the walk returns to $x$ before time $N$.
\begin{proof}[Proof of \cref{thm:inc_trans}]
  We first show the easier claim that $\sum_{t = 0}^{\infty} \mathbb{P}(X_t = x) < \infty$. For this part, we do not need a stopping time and can directly work on the infinite graph. Hence $\mu_t = \nu_t$ is simply the law of $X_t$. Also define
\[
  \Lambda(T) = \sum_{t = 0}^{T-1} \rho_t(x) = \sum_{t = 0}^{T-1} \frac{\mathbb{P}(X_t = x)}{\pi_t(x)}.
\]
  By \cref{lemma:J_bound}, we have $\| J_T \|_r^2 \le 2\Lambda(T)$. Let $\psi_0$ be the voltage function corresponding to the unit current flow from $x$ on the \emph{initial} graph $G_0$ (which exists since $G_0$ is transient). We know that $\Delta_0 \psi_0 = \mathbf{1}_x$ and $\| \nabla_0 \psi_0 \|_{r_0}^2 = \Reff^{G_0}(x \oto \infty)$. Moreover, by \cref{lemma:Laplacian_flow},
  \[
    \langle J_T, \nabla_0 \psi_0 \rangle_r = \sum_{t = 0}^{T-1} \langle \nabla_t \rho_t , \nabla_0 \psi_0 \rangle_r = \sum_{t = 0}^{T-1} \langle \rho_t , \Delta_0 \psi_0 \rangle = \sum_{t = 0}^{T-1} \langle \rho_t , \mathbf{1}_x \rangle = \sum_{t = 0}^{T-1} \rho_t(x) = \Lambda(T).
  \]
  Therefore,
  \[
    \Lambda(T)^2 = \langle J_T, \nabla_0 \psi_0 \rangle_r^2 \le \| J_T \|_r^2 \| \nabla_0 \psi_0 \|_r^2 \le 2 \Lambda(T)\Reff^{G_0}(x \oto \infty),
  \]
  which implies
  \[    
  \Lambda(T) \le 2\Reff^{G_0}(x \oto \infty).
  \]
  Taking the limit as $T \to \infty$, we obtain
  \[
  \sum_{t = 0}^{\infty} \mathbb{P}(X_t = x) = \sum_{t = 0}^{\infty} \pi_t(x) \rho_t(x) \le \pi_{\infty}(x) \lim_{T \to \infty} \Lambda(T) \le 2\pi_{\infty}(x) \Reff^{G_0} (x \oto \infty).
  \]
  In other words, the expected number of visits to $x$ is bounded by $2 \pi_\infty(x) \Reff^{G_0}(x \leftrightarrow \infty)$, which is finite since $G_0$ is transient.
  The same argument shows that the expected number of visits to any other point is also finite; in particular, starting at $x$ but studying the expected number of visits to $y$, letting $\tau_y$ denote the hitting time of $y$, our previous argument implies
  \[ \sum_{t=0}^\infty\P(X_t=y)\leq 2\pi_\infty(y)\E\bracket{\Reff^{G_{\tau_y}}(y\to\infty)} \leq 2\pi_\infty(y)\Reff^{G_0}(y\oto\infty), \]
  where the second inequality follows from Rayleigh monotonicity.

  Now we proceed to the stronger claim that there exists some positive probability $\delta_x$ of non-return. Our strategy here is to instead consider the random walk on the static graph $G_0$ until some time $N$, after which we transition to the dynamic graph $G_{N + 1}, G_{N + 2}, \dots$. This gives us better control over the probability of returning before time $N$, and changing from this to the true dynamic walk cannot destroy the positive-probability event of non-return. To this end, define the new dynamic graph $G_t^N = (G_0, G_0, \dots , G_0, G_{N + 1}, G_{N + 2}, \dots )$ and the dynamic random walk $X_t^N$ and operators $\nabla_t^N, \Delta_t^N$ with respect to this sequence. We define the stopping time
  \[    
  \sigma_x^N := \inf \{ 1 \le t \le N : X_t^N = x \}
  \]
  to be the first return to $x$ before time $N$ (so $\sigma_x^N = \infty$ if such an event does not occur) and relevant quantities $\mu_t^N, \nu_t^N, \rho_t^N$, and $J_T^N$ as before. Also define
  \[
  q_N(x) := \mathbb{P}(\sigma_x^N = \infty)
  \]
  to be the probability that the walk does not return to $x$ before $N$. We know that
  \[    
  \lim_{N \to \infty} q_N(x) = q_{\infty}(x) := \frac{1}{\pi_0(x) \Reff^{G_0}(x \oto \infty)} > 0.
  \]
  From here, repeat the same argument as before but with the stronger bound of 
  \[
    \| J_T^N \|_2^2 \le q_N(x) \left( 2 \sum_{t = N}^{T-1} \rho_t^N(x) + \frac{1}{\pi_0(x)} \right),
  \]
  which we prove in \cref{lem:J_trans_bound} below; this yields, recalling $\rho_t^N(x)=0$ for all $1 \leq t \leq N$ and $\rho_0^N(x)=\frac{1}{\pi_0(x)}$,
  \[ 
    \paren{\sum_{t=N}^{T-1}\rho_t^N(x) + \frac{1}{\pi_0(x)}}^2 = \Lambda^N(T)^2 \leq q_N(x)\paren{2\sum_{t=N}^{T-1}\rho_t^N(x)+\frac{1}{\pi_0(x)}}\Reff^{G_0}(x\oto\infty).
  \]
  After solving a quadratic equation, this gives us
  \begin{align*}
  \sum_{t = N}^{T-1} \rho_t^N(x) &\le q_N \Reff^{G_0}(x \oto \infty) - \frac{1}{\pi_0(x)} + \sqrt{q_N \Reff^{G_0}(x \oto \infty) \left( q_N \Reff^{G_0}(x \oto \infty) - \frac{1}{\pi_0(x)}\right)} \\
  &= \left(q_N - q_{\infty} + \sqrt{(q_N - q_{\infty}) q_N} \right) \Reff^{G_0}(x \oto \infty),
  \end{align*}
  which converges to zero as $N \to \infty$. The right hand side is independent of $T$, so we can send $T\to\infty$ and conclude that 
  \[    
  \lim_{N \to \infty} \mathbb{P}(X_t^N \ne x \text{ for all }t \ge 1) = \lim_{N \to \infty}q_n(x)\mathbb{P}(X_t^N \neq x \text{ for all } t > N \mid \sigma_x^N = \infty)= q_{\infty}(x),
  \]
  as
  \[ 
    \mathbb{P}(X_t^N \neq x \text{ for all } t > N\mid \sigma_x^N = \infty)
    \geq q_N(x) \paren{1 - \sum_{t=N}^{\infty} \nu_t^N(x)}
    \geq q_{\infty}(x) \paren{1 - \pi_\infty(x) \sum_{t=N}^{\infty}\rho_t^N(x)}
  \]
  where the last sum converges to 0 as $N\to\infty$.
  Thus, the probability $\mathbb{P}(X_t^N \ne x \text{ for all } t \ge 1)$ is strictly positive for sufficiently large $N$. Every finite path possible under $X_t^N$ has a positive probability of occurring for $X_t$, since conductances increase, so conditioning on the first $N$ steps, this positivity transfers to the original walk, i.e., $\mathbb{P}(X_t \ne x \text{ for all } t \ge 1) > 0$.
\end{proof}
\begin{lemma}
\label{lem:J_trans_bound}
    \[
    \| J_T^N \|_r^2 \le q_N(x) \paren{\sum_{t=N}^{T-1}\rho_t^N(x) + \frac{1}{\pi_0(x)}}
    \]
\end{lemma}

\begin{proof}
    Recalling the proof of \cref{lemma:J_bound}, we may write
    \[
    \| J_T^N \|_2^2 = \| J_N^N \|_2^2 + \sum_{t = N}^{T-1} \| \nabla_t \rho_t^N \|_r^2 + 2 \sum_{t = N}^{T-1} \sum_{s = 0}^{t-1} \bra \nabla_s \rho_s^N, \nabla_t \rho_t^N \ket_r.
    \]
    Since the graph is static up to time $N$, by \cref{lemma:Laplacian_flow},
    \[  
    \| J_N^N \|_2^2 = \left\| \sum_{t = 0}^{N - 1} \nabla_0 \rho_t^N \right\|_r^2 = \bra \mu_0 - \mu_N^N , \sum_{t=0}^{N-1} \rho_t^N \ket \le \sum_{t = 0}^{N - 1} (\mu_0(x) - \mu_N^N (x))(\rho_t^N(x)) =  q_N(x)\sum_{t = 0}^{N-1}\rho_t^N(x) = \frac{q_N(x)}{\pi_0(x)},
    \]
    where the inequality follows from recalling $\mu_0=\delta_x$.
    Following the proof of \cref{lemma:J_bound}, the remaining terms are similarly bounded by
    \[
        \sum_{t = N}^{T-1} \langle 2 \mu_0 - \mu_t^N - \mu_{t+1}^N, \rho_t \rangle \le \sum_{t=N}^{T-1} (2\mu_0(x) -\mu_t^N(x) - \mu_{t+1}^N(x))\rho_t(x) \le 2q_N(x) \sum_{t = N}^{T-1} \rho_t(x)
    \]
    since $\mu_t(x) \ge 1 - q_N(x)$ whenever $t \ge N$.
\end{proof}

\section{Continuous time}\label{sec:ct}
Here we prove \cref{thm:ct_inc_rec,thm:ct_inc_trans}. To this end, we generalize the notion of a monotone graph to the continuous-time setting. We assume that the graph does not have any loops, as they do not affect the dynamics of the random walk. Previously, we identified increasing the conductance of an edge with adding a new parallel edge. In continuous time, a natural generalization is to view the edge as a growing interval $C_t(e) = [0, c_t(e)] \subseteq \mathbb{R}$. From this perspective, we can define the derivative $D_t : L^2(V) \to L^2(E \times [0, \infty))$ and Laplacian $\Delta_t$ by
\[
  (D_t f)(e, u) = \1{u \le c_t(e)}(f(e^-) - f(e^+)), \quad (\Delta_t f)(v) = \sum_{e^- = v} \int_0^{\infty} (D_t f)(e, u) \di u = \sum_{e^- = v} c_t(e) \di f(e).
\]
This naturally leads to the inner product
\[
  \langle D_t f, D_s g \rangle_{\mathcal{H}} = \sum_{e \in E} \int_0^{\infty} (D_t f)(e, u) (D_s g)(e, u) \di u = \sum_{e \in E} \min(c_t(e), c_s(e)) (\di f)(e) (\di g)(e)
\]
where $\mathcal{H} = L^2(E \times [0, \infty))$. Now if $G_s \le G_t$, then 
\[
\langle D_s f, D_t g \rangle_{\mathcal{H}} = \langle D_s f, D_s g \rangle_{\mathcal{H}} = \langle \Delta_s f, g \rangle.
\]

Recall the definitions of the measures $\mu_t$ and $\nu_t$ from the discrete-time setting. In continuous time, the main difference is the following analog of \eqref{eq:laplacian}, which requires using $\nu_t$ instead of $\rho_t$.
This is simply because the instantaneous jump rate is $c_t(e)$, rather than $\frac{c_t(e)}{\pi_t(e^-)}$ as in the discrete-time setting. In order to circumvent the possibility of explosions, i.e., infinitely many jumps in a finite time interval, we restrict to stopping times that restrict the walk to a finite domain $B$. Since the jump rate is bounded by the maximum conductance over all vertices in $B$, which is finite, this ensures that explosions do not occur before this stopping time.
\begin{lemma}
Let $\tau$ be a stopping time such that $X_{t \wedge \tau}$ is contained in a finite set. Define $\mu_t$ and $\nu_t$ as
\[  
\mu_t(v) = \mathbb{P}(X_{t \wedge \tau} = v), \quad \nu_t(v) = \mathbb{P}(X_t = v , t < \tau).
\]
Then,
  \[
  \mu_0(v) - \mu_t(v) = \int_0^t \sum_{e^- = v} c_s(e)(\nu_s(v) - \nu_s(e^+)) \di s.
  \]
In other words,
\[
\Delta_t \nu_t = - \frac{\di}{\di t} \mu_t.
\]  
\end{lemma}
\begin{proof}
  Let $N_t(e)$ be the number of jumps $X_{s \wedge \tau}$ takes along edge $e$ up to time $t$. Its expected value is given by
  \[
  \mathbb{E} N_t(e) = \int_0^t c_s(e) \nu_s(e^-) \di s.
  \]
  On the other hand, by counting the number of times $X_{s \wedge \tau}$ leaves and enters $v$, we have
  \[
  \1{X_0 = v} - \1{X_{t \wedge \tau} = v} = \sum_{e^- = v} N_t(e) - \sum_{e^+ = v} N_t(e).
  \]
  Thus by taking expectations and using symmetry, we have
  \[
  \mu_0(v) - \mu_t(v) = \int_0^t \sum_{e^- = v} c_s(e)\nu_s(v) \di s - \int_0^t \sum_{e^+ = v} c_s(e)\nu_s(e^-) \di s = \int_0^t \sum_{e^- = v} c_s(e)(\nu_s(v) - \nu_s(e^+)) \di s. \qedhere
  \]
\end{proof}
From this point, the rest of the proof carries over from the discrete-time setting almost verbatim after replacing $\rho_t$ with $\nu_t$. For example, we have the following analog of \cref{lemma:J_bound}:
\[
  \left\| \int_0^t D_s \nu_s \di s \right\|_{\mathcal{H}}^2 \le 2 \int_0^t \nu_s(x) \di s.
\]
The only concern is the explosion time $\zeta \in [0 , \infty]$, which we treat in the proofs below. In both cases, define the stopping times
\[  
\sigma_B := \inf \{ t \ge 0 : X_t \notin B \}
\]
to be the first exit time. This ensures that $X_{t \wedge \sigma_B}$ makes finitely many steps and that $\sigma_B \uparrow \xi$ as $B \uparrow V$.
\begin{proof}[Proof of Theorem~\ref{thm:ct_inc_rec}]
    Repeat the proof of Theorem~\ref{thm:inc_rec} with $\tau_B = \sigma_B \wedge \tau_x$ where
    \[  
    \tau_x := \inf \{ t > 0 : X_t = x , X_s \ne x \text{ for some } s < t  \}\]
    and the modifications mentioned above. The same proof gives
    \[  
    \mathbb{P}(\tau_x > \sigma_B) \le \frac{\sqrt{2}}{\sqrt{\pi_0(x) \Reff^{G}(x \oto B^c)}}.
    \]
    Sending $B \uparrow V$, this implies $\mathbb{P}(\tau_x < \xi) = 1$, or in other words, $X_t$ returns to $x$ before the explosion time. By the Markov property, this also implies infinitely many returns to $x$ before $\xi$. Since the jump rate at $x$ is bounded above by $\pi_{\infty}(x)$, it follows that $\xi$ must be infinite.
\end{proof}

\begin{proof}[Proof of Theorem~\ref{thm:ct_inc_trans}]
    The same proof as Theorem~\ref{thm:inc_trans} shows that
    \[  
    \mathbb{E} \left[ \int_0^{T \wedge \sigma_B} \mathbf{1}_{\{ X_t = x \}} \di t \right] \le 2 \Reff^{G_0} (x \oto B^c).
    \]
    Taking $T \to \infty$ and $B \uparrow V$, we have
    \[  
    \mathbb{E} \left[ \int_0^{\xi} \mathbf{1}_{\{ X_t = x \}} \di t \right] \le 2 \Reff^{G_0} (x \oto \infty).
    \]
    In other words, the expected local time at $x$ before explosion is finite.
    Since the jump rate at $x$ is bounded by $\pi_\infty(x)$, this also ensures the expected number of visits is finite.
\end{proof}

\section{A counterexample for decreasing transience}\label{sec:dec_trans}
In this section we prove \cref{thm:dec_trans}.
Take vertex set $\Z$ and connect adjacent integers with edge conductance $c(n, n+1) = 2^n$, which causes positive drift.
For $n \geq 1$, also connect $n+1$ and $-n$ with conductance $4^n$, to be deleted at time $T_n = n^2$ (i.e., $c(n+1, -n) = 4^n \1{t < T_n}$.
This essentially folds $\Z$ into a doubled half-line; see \cref{fig:dec_trans} for a schematic of the construction.

\begin{figure}[htbp]
\centering
\begin{tikzpicture}[
  x=1cm,y=1cm,>=Latex,line cap=round,line join=round,
  rail/.style={line width=.8pt},
  rung/.style={line width=.8pt,dash pattern=on 3pt off 2pt},
  vertex/.style={circle,fill=black,inner sep=1.6pt},
  elabel/.style={font=\small,fill=white,inner sep=2pt},
  rlabel/.style={font=\footnotesize,fill=white,inner sep=3pt,align=center}
]
% Fold the ordinary integer line at the permanent edge {0,1}.
\foreach \name/\x in {a/0.5,b/3.0,c/5.5,n/9.2,p/12.1}{
  \coordinate (U\name) at (\x,1.15);
  \coordinate (V\name) at (\x,-1.15);
}
\draw[rail] (Ua) -- node[elabel,left] {$1$} (Va);
\draw[rail] (Ua) -- node[elabel,above] {$2$} (Ub);
\draw[rail] (Ub) -- node[elabel,above] {$2^2$} (Uc);
\draw[rail] (Va) -- node[elabel,below] {$2^{-1}$} (Vb);
\draw[rail] (Vb) -- node[elabel,below] {$2^{-2}$} (Vc);
\draw[rail] (Uc) -- (6.45,1.15);
\draw[rail] (8.25,1.15) -- (Un);
\draw[rail] (Vc) -- (6.45,-1.15);
\draw[rail] (8.25,-1.15) -- (Vn);
\node at (7.35,1.15) {$\cdots$};
\node at (7.35,-1.15) {$\cdots$};
\node at (7.35,0) {$\cdots$};
\draw[rail] (Un) -- node[elabel,above] {$2^{n+1}$} (Up);
\draw[rail] (Vn) -- node[elabel,below] {$2^{-(n+1)}$} (Vp);
\draw[rail,->] (Up) -- (14.0,1.15);
\draw[rail,->] (Vp) -- (14.0,-1.15);
\node[font=\small,above=3pt] at (13.75,1.15) {$+\infty$};
\node[font=\small,below=3pt] at (13.75,-1.15) {$-\infty$};
\draw[rung] (Ub) -- node[rlabel] {$4 \cdot \1{t < T_1}$} (Vb);
\draw[rung] (Uc) -- node[rlabel] {$4^2 \cdot \1{t < T_2}$} (Vc);
\draw[rung] (Un) -- node[rlabel] {$4^n \cdot \1{t < T_n}$} (Vn);
\draw[rung] (Up) -- node[rlabel] {$4^{n+1} \cdot \1{t < T_{n+1}}$} (Vp);
\foreach \name/\lab in {a/1,b/2,c/3,n/{n+1},p/{n+2}}{
  \node[vertex,label=above:{$\lab$}] at (U\name) {};
}
\foreach \name/\lab in {a/0,b/{-1},c/{-2},n/{-n},p/{-(n+1)}}{
  \node[vertex,label=below:{$\lab$}] at (V\name) {};
}
\draw[->,line width=.65pt] (3.5,2.05) -- (10.5,2.05)
  node[midway,above=2pt,font=\small] {limiting drift toward $+\infty$};
\draw[->,line width=.65pt] (10.5,-2.05) -- (3.5,-2.05)
  node[midway,below=2pt,font=\small] {limiting drift toward $+\infty$};
\end{tikzpicture}
\caption{The full integer line folded at $\{ 0, 1 \}$. Dashed vertical edges join $n+1$ to $-n$ and are deleted at $T_n=n^2$.}
\label{fig:dec_trans}
\end{figure}
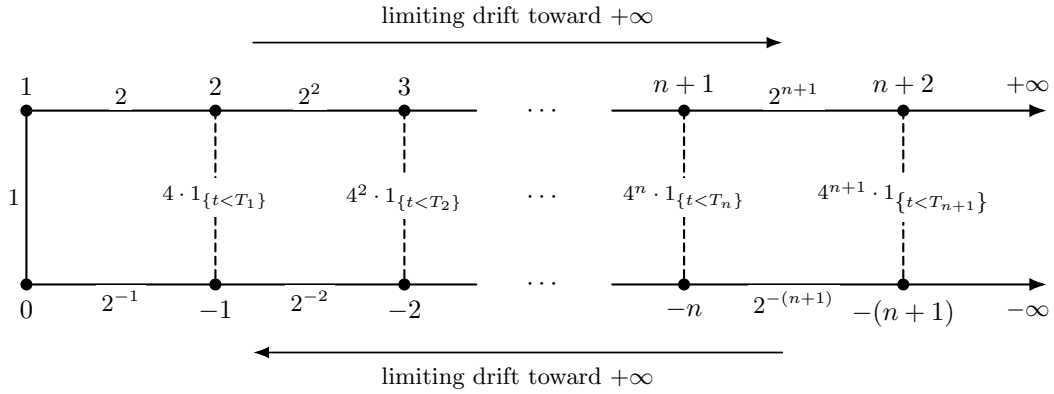

This is a decreasing network where the limiting graph is $\mathbb{Z}$ with edge conductances $c_\infty(n, n+1) = 2^n$.
Since this is simply the biased random walk on $\mathbb{Z}$ that moves right with probability $\frac23$, we know it is transient.
The additional key observation is that the vertical edges dominate the total conductance, and once they are deleted, the remaining graph is a biased random walk on the integers.
Thus, the dominant behavior of the dynamic walk $X_t$ (at least on a heuristic level) is the following.
\begin{enumerate}[label = \arabic*.]
    \item At time $T_{n-1} = (n-1)^2$, $X_{T_{n-1}} =-n + 1$.
    \item From here, $X_t$ either hits $-n$ or follows the positive drift towards $n+1$. These two events happen with roughly equal probability.
    \item Once $X_t$ reaches $n+1$ or $-n$, it moves back and forth between these two points along the vertical edge $e_n$.
    \item The edge $e_n$ gets deleted at time $T_n$. Since $X_t$ alternates between even and odd numbers (so $t$ and $X_t$ always have the same parity), we have $X_{T_n} = -n$ and the process repeats itself.
\end{enumerate}

The remainder of this section is devoted to proving that the above scenario indeed occurs infinitely often, starting with $X_0=0$.
The second step is easy to justify, as we show in the following lemma.
\begin{lemma}
    For all $n\geq 1$,
    \[  
    \mathbb{P}(\tau_{n+1} < \tau_{-n} | X_{T_{n-1}} = -n + 1) = \frac{2^{2n}}{2^{2n+1} - 1} > \frac{1}{2}
    \]
    where both hitting times are measured from $T_{n-1}$.
\end{lemma}

\begin{proof}
Conditioned on $X_{T_{n-1}} = -n + 1$, $X_{T_{n-1} + t}$ follows a biased nearest-neighbor random walk on $[-n, n+1]$ until it hits an endpoint. The probability of moving right is $\frac23$ at each step, so a gambler's ruin computation shows that
\[  
    \mathbb{P}(\tau_{n+1} < \tau_{-n} | X_{T_{n-1}} = -n + 1) = \frac{2^{2n}}{2^{2n+1} - 1}. \qedhere
\]
\end{proof}

Since $\tau_{n+1} < \tau_{-n}$ implies that $X_t$ passes through $0$, it only remains to show that $X_{T_n} = -n$ infinitely often; as in the proof of \cref{thm:inc_rec}, knowing $X_t$ visits 0 infinitely often implies $X_t$ visits every vertex infinitely often, as $c_\infty$ is connected.
We will categorize the failure modes into two situations: the first is that the walk does not encounter the vertical edge $e_n$ at all, and the second is that it does not travel along $e_n$ even when given the opportunity.
The next two lemmas show that each of these events can be avoided.

\begin{lemma}
    $X_t$ encounters infinitely many vertical edges.
\end{lemma}

\begin{proof}
    Suppose not. Then there exists some last vertical edge $e_N$ that $X_t$ encounters. This has two implications:
    \begin{itemize}
        \item $\{ X_t : t \ge T_N \}$ behaves as the biased random walk on $\mathbb{Z}$.
        \item $\abs{X_t} < \sqrt{t}+1$ whenever $t \ge T_N$, since otherwise $X_t$ would encounter a vertical edge $e_{N'}$ for $N'>N$.
    \end{itemize}
    These two facts are contradictory, since the biased random walk on $\mathbb{Z}$ has a linear drift, thus
    \[  
    \lim_{t \to \infty} \frac{X_t}{t} = \frac{1}{3},
    \]
    while the second implication requires sublinear growth.
\end{proof}

\begin{lemma}
    Let $A_n = \{ -n ,n+1 \}$ and define
    \[  
    B_n = \{ X_t \in A_n, X_{t+1} \notin A_n \text{ for some } 0 \le t < T_n \}.
    \]
    Then, $B_n$ occurs finitely often.
\end{lemma}

\begin{proof}
    The horizontal edge conductances on $-n$ and $n+1$ are bounded by $2^{n+1}$, so we may bound
    \[  
    \sum_{n=1}^{\infty} \mathbb{P}(B_n) \le \sum_{n=1}^{\infty} \sum_{t=0}^{n^2-1} \frac{2 \cdot 2^{n+1}}{4^n + 2 \cdot 2^{n+1}} \le \sum_{n=1}^{\infty} \frac{n^2}{2^{n-2}} < \infty.
    \]
    Hence by the first Borel--Cantelli lemma, $B_n$ happens finitely often.
\end{proof}

Combining the above lemmas, we see that $X_t$ encounters infinitely many vertical edges $e_n$, of which all but finitely many lead to $X_{T_n} = -n$, as always traversing $e_n$ once encountered ensures by parity that $X_{T_n}=-n$.
This completes the proof of \cref{thm:dec_trans}.

\begin{remark}
    It is also possible to construct examples where we do not send conductances to zero. Indeed, our example can be altered so that the vertical conductances do not vanish, but rather drop to some sufficiently small value, e.g., $2^{-4n}$. However, we have not been able to find an example where the conductances drop only by a constant factor and leave this as an open question; see Question~\ref{question:dec_trans_alpha}.
\end{remark}

\section{A counterexample for decreasing recurrence}\label{sec:dec_rec}
In this section we prove \cref{thm:dec_rec}.

\subsection{Construction}
Fix the following parameters: for $k\geq 1$ and $d\geq 0$,
\begin{equation*}\begin{aligned}
    a_k &= k^{-5/2} \\
    \vartheta &= \frac54 \\
    V(d) &= (d+1)^\vartheta \\
    A(d) &= \sum_{k=1}^\infty a_k\paren{V(d+k)-V(d)} + \sum_{k=1}^d a_k\paren{V(d-k)-V(d)} \\
    R(d) &= \sum_{k=d+1}^\infty a_k.
\end{aligned}\end{equation*}
The choice of $\frac52$ is not essential; more generally, one may take $a_k=k^{-1-\beta}$ whenever $1<\vartheta<\beta<2$, but we shall proceed with these parameters.
The following bounds will be useful.
\begin{lemma}\label{lemma:dec_rec_bounds}
There are absolute constants $K,c,\alpha_*>0$ with $\alpha_*\leq1$ such that, for every $d\in\N$ and $0<\alpha\leq\alpha_*\leq1$,
\begin{align}
  A(d)&\leq K(d+1)^{-1/4},
  \label{eq:dec_rec_A_bound}\\
  R(d)&\geq c(d+1)^{-3/2},
  \label{eq:dec_rec_R_bound}\\
  \alpha A(d)-V(d)R(d)&\leq 0.
  \label{eq:dec_rec_full_superharmonic}
\end{align}
\end{lemma}
\begin{proof}
Throughout the proof, $c$ will denote a sufficiently small absolute constant, and $K$ and $C$ will denote sufficiently large absolute constants; their values may change from line to line.
We have
\[
  R(d)\geq \int_{d+1}^\infty x^{-5/2}\di x
  =\frac23(d+1)^{-3/2},
\]
which implies \eqref{eq:dec_rec_R_bound}.

To prove \eqref{eq:dec_rec_A_bound}, assume $k\geq 3$ and we split the sums into three ranges.
For $1\leq k\leq d/2$, Taylor approximation gives
\[
  V(d+k)+V(d-k)-2V(d)
  \leq C k^2(d+1)^{-3/4}.
\]
Hence this range contributes at most
\[
  C(d+1)^{-3/4}\sum_{k\leq d/2}k^{-1/2}
  \leq C(d+1)^{-1/4}.
\]
For $d/2<k\leq d$, ignoring the negative terms, the summand is at most
$C(d+1)^{5/4}$, while
\[
  \sum_{d/2<k\leq d}k^{-5/2}\leq C(d+1)^{-3/2},
\]
which combines to also yield a contribution bounded by $C(d+1)^{-1/4}$.
Finally, for $k>d$,
\[
  V(d+k)-V(d)\leq V(2k-1) \leq Ck^{5/4},
\]
so
\[
  \sum_{k>d}k^{-5/2}\bigl(V(d+k)-V(d)\bigr)
  \leq C\sum_{k>d}k^{-5/4}
  \leq C(d+1)^{-1/4}.
\]
This proves \eqref{eq:dec_rec_A_bound}.

Since $V(d)=(d+1)^{5/4}$, \eqref{eq:dec_rec_R_bound} implies $V(d)R(d)\geq c(d+1)^{-1/4}$, and thus \eqref{eq:dec_rec_full_superharmonic} holds whenever $\alpha K\leq c$.
\end{proof}
The construction will use the following lemma to attach holding rooms to a random walk on $\N$ to synchronize the random walk to the times when $(c_t)_{t\geq0}$ changes, serving a similar but more complicated role as the bipartite parity effect for the timing of the construction in \cref{sec:dec_trans}, i.e., why $X_{T_n}=-n$ instead of $X_{T_n}=n+1$.
A \emph{pendant room} at the base vertex $n$ is a new vertex $q_n$ joined only to $n$, by an
edge of conductance $M_n$.
We extend every test function from $n$ to $q_n$ by giving the two vertices the same value.
Thus, these pendant rooms do not contribute to the Dirichlet energy, and only delay the walk.

\begin{lemma}\label{lemma:dec_rec_sync}
Consider a base sequence of conductances $(c_t)_{t\geq0}$ on vertex set $\N$.
Suppose the total conductance is at most $D$ at every vertex, uniformly over all time, i.e., $\pi_t(v)\leq D$ for all $t\geq0$ and $v\in\N$.
Suppose that for each $n\in\N$, in the static network $(N,c_n)$, the walk started from $n$ hits $[n+1,\infty)$ almost surely.
Then given positive numbers $(\eps_n)_{n\geq0}$ with $\sum_{n=0}^\infty \eps_n<\frac12$, one can choose
finite room conductances $M_n>0$ for $n\geq 1$ and deterministic integers
\[
  0=T_0<T_1<T_2<\cdots
\]
so that the following holds.
Consider the random walk $(X_t)_{t\geq0}$ in changing environment on extended vertex set $\N\cup\set{q_1,q_2,\dots}$ where pendant room $q_n$ is connected only to base vertex $n$, by an edge of conductance $M_n$, constant across all times; the network within $\N$ at time $t$ is given by conductance $c_n$ for $t\in[T_n,T_{n+1})$.
Then with probability at least $1-2\sum_{n=0}^\infty \eps_n$, the following holds for $(X_t)_{t\geq0}$:
\begin{enumerate}[label=\textup{(\roman*)}]
    \item if $n$ is reached before $T_n$, the walk remains in
          $\{n,q_n\}$ until time $T_n$;
    \item starting from $n$ or $q_n$ at time $T_n$, the walk reaches a base
          vertex larger than $n$ before time $T_{n+1}$;
\end{enumerate}
Under this event, the largest base vertex in $\N$ that the walk has reached tends to infinity.
\end{lemma}
\begin{proof}
Choose the parameters recursively.
Having chosen $T_n$ for some $n\geq0$, take $M_n \geq \frac{DT_n}{\eps_n}$.
Whenever the walk is at $n$, the probability of using a base edge instead of the room edge is at most $\frac{D}{M_n}$.
Thus, conditional on reaching $n$ before $T_n$, the probability of leaving $\{n,q_n\}$ before $T_n$ is at
most $\eps_n$.
For $n=0$, we do not need to pick an $M_0$, and as $T_0=0$ this part is vacuously true.

In $(\N,c_n)$ the hitting time of $[n+1,\infty)$ is finite, by assumption, where identifying $[n+1,\infty)$ with a single vertex $\dagger$ reduces the network to a finite network, with all vertices connected to $\dagger$.
Thus, attaching the pendant rooms $q_1,q_2,\dots,q_n$ means the resulting network is still finite and connected, so that the hitting time of $[n+1,\infty)$ for the phase-$n$ network of $(X_t)_{t\geq0}$, i.e., the static network at time $t=T_n$, is finite, regardless of the starting location in $\set{0,1,\dots,n,q_1,q_2,\dots,q_n}$.
Hence, choose finite $L_n$ such that, starting from either $n$ or $q_n$, the probability of not hitting $[n+1,\infty)$ by time $L_n$ is at most $\eps_n$.
Set $T_{n+1}=T_n+L_n+1$.
The result then immediately follows.
\end{proof}
We now provide our counterexample construction for \cref{thm:dec_rec}.
Let the base vertex set be $\N$.  Join every pair $i<j$ by one edge of
initial conductance $c_0(i,j)=a_{j-i}=(j-i)^{-5/2}$.
Attach to each $n\in\N$ one pendant room $q_n$, with static conductance $c_t(n,q_n)=M_n$ for all $t\in\N$, to obtain vertex set $V=\N\sqcup\set{q_1,q_2,\dots}$.
Note that while the base degree is infinite, the total base conductance is uniformly bounded:
\[
  \sum_{j\neq i}c_0(i,j)
  \leq 2\sum_{k\geq1}k^{-5/2}
  =:D<\infty.
\]
Fix $\delta>0$ from the statement of \cref{thm:dec_rec}, and choose positive $(\eps_n)_{n\geq0}$ so that $2\sum_{n=0}^\infty \eps_n<\delta.$
Fix $\alpha<\alpha_*$ from \cref{lemma:dec_rec_bounds} and for $i<j$, let
\[ c_t(i,j) = \begin{cases} a_{j-i} & \text{if }t<j \\ \alpha a_{j-i} &\text{if }t\geq j. \end{cases} \]
Then let $(T_n)_{n\geq0}$ and $(M_n)_{n\geq0}$ from \cref{lemma:dec_rec_sync}, with random walk in changing environment $(X_t)_{t\geq0}$.
For simplicity of notation, henceforth forget the original definition of $(c_t)_{t\geq0}$ and let $(c_t)_{t\geq0}$ now denote the environment of $(X_t)_{t\geq0}$.
Then the edge $\{i,j\}$ for $i<j$ decreases conductance once, at time $T_j$:
\[
  c_t(i,j)
  =
  \begin{cases}
    a_{j-i} & \text{if }t<T_j,\\
    \alpha a_{j-i} & \text{if }t\geq T_j.
  \end{cases}
\]
Thus $c_t\downarrow c_\infty$ where $c_\infty(i,j)=\alpha c_0(i,j)$ for $i<j$ and $c_\infty(n,q_n)=c_0(n,q_n)$, so in particular $c_\infty\geq\alpha c_0$.

See \cref{fig:dec_rec} for a schematic of the construction and its intuition.
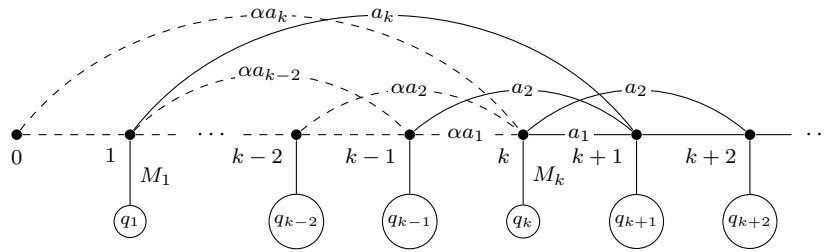
\begin{figure}[htbp]
\centering
\begin{tikzpicture}[
  scale=1,
  base/.style={circle,fill,inner sep=1.4pt},
  room/.style={circle,draw,inner sep=1.2pt},
  scaled/.style={dashed},
  lbl/.style={font=\small},
  clbl/.style={font=\footnotesize},
  albl/.style={font=\footnotesize,fill=white,inner sep=1pt}
]
% ---------- base vertices ----------
\node[base,label={[clbl]below:$0$}] (v0) at (0,0) {};
\node[base,label={[clbl]below left:$1$}] (v1) at (1.5,0) {};
\node[lbl] at (2.6,0) {$\cdots$};
\node[base,label={[clbl]below left:$k-2$}] (km2) at (3.7,0) {};
\node[base,label={[clbl]below left:$k-1$}] (km1) at (5.2,0) {};
\node[base,label={[clbl]below left:$k$}] (vk) at (6.7,0) {};
\node[base,label={[clbl]below left:$k+1$}] (kp1) at (8.2,0) {};
\node[base,label={[clbl]below left:$k+2$}] (kp2) at (9.7,0) {};
\node[lbl] at (10.65,0) {$\cdots$};
% ---------- nearest-neighbor base edges ----------
\draw[scaled] (v0) -- (v1);
\draw[scaled] (v1) -- (2.2,0) (3.0,0) -- (km2);
\draw[scaled] (km2) -- (km1);
\draw[scaled] (km1) -- node[albl,pos=0.5] {$\alpha a_1$} (vk);
\draw (vk) -- node[albl,pos=0.5] {$a_1$} (kp1);
\draw (kp1) -- (kp2);
\draw (kp2) -- (10.25,0);
% ---------- longer-range edges (a few representatives) ----------
\draw[scaled] (km2) to[bend left=40] node[albl,pos=0.5] {$\alpha a_2$} (vk);
\draw[scaled] (v1) to[bend left=45] node[albl,pos=0.5] {$\alpha a_{k-2}$} (km1);
\draw[scaled] (v0) to[bend left=52] node[albl,pos=0.5] {$\alpha a_{k}$} (vk);
\draw (km1) to[bend left=40] node[albl,pos=0.5] {$a_2$} (kp1);
\draw (vk) to[bend left=40] node[albl,pos=0.5] {$a_2$} (kp2);
\draw (v1) to[bend left=52] node[albl,pos=0.5] {$a_{k}$} (kp1);
% ---------- pendant rooms ----------
\node[room] (q1) at (1.5,-1.15) {\scriptsize $q_1$};
\node[room] (qm2) at (3.7,-1.15) {\scriptsize $q_{k-2}$};
\node[room] (qm1) at (5.2,-1.15) {\scriptsize $q_{k-1}$};
\node[room] (qk) at (6.7,-1.15) {\scriptsize $q_{k}$};
\node[room] (qp1) at (8.2,-1.15) {\scriptsize $q_{k+1}$};
\node[room] (qp2) at (9.7,-1.15) {\scriptsize $q_{k+2}$};
\draw (v1) -- node[right,clbl,pos=0.55] {$M_1$} (q1);
\draw (km2) -- (qm2);
\draw (km1) -- (qm1);
\draw (vk) -- node[right,clbl,pos=0.55] {$M_k$} (qk);
\draw (kp1) -- (qp1);
\draw (kp2) -- (qp2);
\end{tikzpicture}
\caption{A schematic of the construction, shown during phase $k$, i.e., the time interval $[T_k,T_{k+1})$.
The base is the complete graph on $\N$, with initial conductances $c_0(i,j)=a_{|j-i|}=|j-i|^{-5/2}$, and each base vertex $n\geq1$ carries a pendant room $q_n$, joined to $n$ alone by an edge of static conductance $M_n$.
The base edge $\set{i,j}$ with $i<j$ is scaled down exactly once, at time $T_j$, from $a_{j-i}$ to $\alpha a_{j-i}$; thus, during phase $k$, the dashed edges have been scaled down by $\alpha$, while the solid base edges still carry their initial conductances.
The rooms synchronize the walk with the schedule: upon reaching a new maximum $n$ (before time $T_n$), the walk is held in $\set{n,q_n}$ until time $T_n$.
When it is released, in phase $k\geq n$, every base edge leading backward has been scaled by $\alpha$ while the edges to $[k+1,\infty)$ have not; this imbalance pushes the walk to a new maximum, and these maxima tend to infinity almost surely, even though every network $(V,c_t)$, $t\in[0,\infty]$, is recurrent.}
\label{fig:dec_rec}
\end{figure}
\begin{lemma}\label{lemma:dec_rec_rec}
Every network $(V,c_t)$ for $t\in[0,\infty]$ is recurrent.
\end{lemma}
\begin{proof}
By Rayleigh monotonicity, it suffices to show $(V,c_0)$ is recurrent.
For $N\geq1$, define $f_N$ by
\[
  f_N(n)=\paren{1-\frac{n}{N}}_+ = \max\set{1-\frac{n}{N},0},
\]
and recall we extend functions on $\N$ to $V$ via the convention $f_N(q_n)=f_N(n)$ for all $n\in\N$.
Then the Dirichlet energy is given by
\begin{align*}
  \mc{E}_{c_0}(f_N)
  &=\sum_{k=1}^\infty k^{-5/2}
    \sum_{i=0}^\infty \paren{f_N(i+k)-f_N(i)}^2\\
  &\leq \frac1N\sum_{k=1}^N k^{-1/2}
       +N\sum_{k=N+1}^\infty k^{-5/2}\\
  &\leq C N^{-1/2}
\end{align*}
for some large constant $C$, bounding $f_N(i+k)-f_N(i)\leq \frac{k}{N}$ in the first sum across $0 \leq i \leq N-1$, and $f_N(i+k)-f_N(i)\leq 1$ in the second sum across $0\leq i \leq N-1$.
As $f_N(0)=1$ and $f_N(N)=0$, Dirichlet's principle from \cref{subsec:review} implies $\Reff^{c_0}(0\oto N) \geq \frac1C\sqrt{N}$, and thus $\Reff^{c_0}(0\oto\infty)=\infty$, so $c_0$ is recurrent.
\end{proof}

\subsection{Transience}

For $n\geq3$, let $D_n = \floor{\frac{n}3}$ and let $m_n=n-D_n$.
Let $\P_v^{(n)}$ denote the law of the walk in the frozen phase-$n$ environment of $(X_t)_{t\geq0}$, started from the vertex $v\in V$; recall this walk includes the pendant rooms.
For a set $A\subseteq V$, let $\tau_A$ denote the hitting time of $A$.
\begin{lemma}\label{lemma:dec_rec_front}
For every $n\geq3$,
\[
  \P_n^{(n)}
  \left(
    \tau_{[0,m_n]}<\tau_{[n+1,\infty)}
  \right)
  \leq (D_n+1)^{-5/4}.
\]
The same bound holds when the walk starts from $q_n$, i.e., for $\P_{q_n}^{(n)}$.
\end{lemma}
\begin{proof}
Let $(Y_t)_{t\geq0}$ have law $\P_n^{(n)}$.
As $q_n$ is only connected to $n$, the result for $q_n$ immediately follows from the result for $n$.
Define
\[
  \Phi_n(i)=
  \begin{cases}
    V(D_n) & \text{if }i\leq m_n,\\
    V(n-i) & \text{if }m_n<i\leq n,\\
    0 & \text{if }i>n,
  \end{cases}
\]
with $\Phi_n(q_i)=\Phi_n(i)$ for all $i\in\N$.
We claim $\Phi_n(Y_{t\wedge\tau})$ is a bounded supermartingale for $\tau=\tau_{[0,m_n]}\wedge\tau{[n+1,\infty)}$.
Consider $i=n-d$ with $0\leq d<D_n$.
All jumps to $j\leq n$ have their conductance multiplied by $\alpha$, whereas every jump of length $k>d$ to a vertex $j>n$ has unswitched conductance $a_k$ and lands where $\Phi_n=0$.
Thus, conditioned on $\set{Y_t=i}\cap\set{t<\tau}$,
\begin{align*}
    &\E\bracket{\Phi_n(Y_{t+1\wedge\tau})-\Phi_n(Y_{t\wedge\tau})\middle|Y_t=i,t<\tau}
    \\ &= \frac{1}{\pi_n(i)}\paren{\sum_{k=1}^{n-d}\alpha a_k\paren{\Phi_n(i-k)-\Phi_n(i)} + \sum_{k=1}^d\alpha a_k\paren{\Phi_n(i+k)-\Phi_n(i)} - \sum_{k=d+1}^\infty a_k\Phi_n(i)}
    \\ &\leq \frac{1}{\pi_n(i)}\paren{\sum_{k=1}^{n-d}\alpha a_k\paren{V(d+k)-V(d)}+\sum_{k=1}^d\alpha a_k\paren{V(d-k)-V(d)} - \sum_{k=d+1}^\infty a_k V(d)}
    \\ &\leq \frac{1}{\pi_n(i)}\paren{\alpha A(d) - V(d)R(d)}
    \\ &\leq 0
\end{align*}
by \eqref{eq:dec_rec_full_superharmonic}, where $\pi_n$ denotes the total conductance of the frozen phase-$n$ environment of $(X_t)_{t\geq0}$.
As $\tau<\infty$ almost surely, because the interior is finite, optional stopping gives
\[
  1=\Phi_n(n)
  \geq \E\bracket{\Phi_n(Y_\tau)}
  = V(D_n)
  \P_n^{(n)}\paren{\tau_{[0,m_n]}<\tau_{[n+1,\infty)}}.
\]
Since $V(D_n)=(D_n+1)^{5/4}$, this proves the claim.
\end{proof}

Define $\Base:V\to\N$ via $\Base(n)=\Base(q_n)=n$ for all $n\in\N$.
To prove \cref{thm:dec_rec}, it suffices to show that $\Base(X_t)\to\infty$ with probability at least $1-\delta$, as this implies every vertex is visited only finitely often.

Let $G$ be the event from \cref{lemma:dec_rec_sync}, so that $\P(G)\geq1-\delta$.
On $G$, write
\[
  0=N_0<N_1<N_2<\cdots
\]
for the distinct values of $\max_{t\leq T}\Base(X_t)$ across $T\in\N$.
We call $\set{N_0,N_1,\dots}$ the set of \emph{records} of $(X_t)_{t\geq0}$.

For each $n\geq3$, let $A_n$ be the event that $n$ is a record and, from the time $S_n$ that $(X_t)_{t\geq0}$ first reaches $n$, the base coordinate reaches $[0,m_n]$ before the next record is found, i.e., before reaching $[n+1,\infty)$.
On $G$, we have $S_n<T_n$, as some record in $[n+1,\infty)$ must be made before time $T_n$, and $\Base(X_t)=n$ for all $t\in[S_n,T_n]$.
We can naturally couple $(X_t)_{t\in[T_n,T_{n+1})}|X_{T_n}=n$ with $(Y_{t-T_n})_{t\in[T_n,T_{n+1})}\sim\P_n^{(n)}$ and similarly couple $(X_t)_{t\in[T_n,T_{n+1})}|X_{T_n}=q_n$ with $(Y_{t-T_n})_{t\in[T_n,T_{n+1})}\sim\P_{q_n}^{(n)}$.
Then on $G$, the corresponding $(Y_{t-T_n})_{t\in[T_n,T_{n+1})}$ hits $[0,m_n]\cup[n+1,\infty)$ before the process stops at time $t=T_{n+1}$, so
\[ \P(A_n|G) \leq (D_n+1)^{-5/4} \]
by \cref{lemma:dec_rec_front}.
Since
\[
  \sum_{n\geq3}(D_n+1)^{-5/4}<\infty,
\]
the first Borel--Cantelli lemma implies that almost surely on $G$, only finitely many $A_n$ occur.
On $G$, after the last such $A_n$ occurs, say $A_{n_0}$, when the current record is $n>n_0$, we have $\Base(X_t)\geq m_n\geq\frac{2n}{3}$.
As the records $N_i$ tend to infinity, it follows that $\Base(X_t)\to\infty$ on $G$, completing the proof of \cref{thm:dec_rec}.

\begin{remark}
    Our construction is not locally finite: every base vertex has infinite degree.
    One can modify our construction, e.g., by restricting the base edges $\set{i,j}$ for $i<j$ only to $j \leq 2i+1$, to be locally finite, but the resulting analysis is more tedious so for simplicity, we have presented this locally infinite construction.
    It is unclear to us whether a counterexample with uniformly bounded degree or uniformly bounded total conductance, i.e., $\pi_t(v) \leq M$ for all $t$ and $v$, exists; note that in our case, $M_n\to\infty$.
\end{remark}

\section{Open questions}\label{sec:open}
Dembo, Huang, and Sidoravicius~\cite{DHS2014} were interested in a broader notion of monotonicity of recurrence and transience of random walks in monotone non-decreasing environments (where they only considered the graph setting, i.e., $\set{0,1}$-valued conductances, so monotone non-decreasing can be more simply thought of as growing, in the sense of edge sets), a generalized Rayleigh monotonicity principle where both environments can be monotone non-decreasing.
We were unable to establish this conjecture but find it very interesting, and restate it below.
\begin{conjecture}[\protect{\cite[Conjecture 1.8]{DHS2014}}]\label{conj:monotone_trans}
    Let $(G_t)_{t\geq0}$ and $(G_t')_{t\geq0}$ be two deterministic sequences of non-decreasing graphs of uniformly bounded degrees on the same countable vertex set $V$, with $G_t \subseteq G_t'$ for all $t$.
    Let $(X_t)_{t\geq0}$ and $(Y_t)_{t\geq0}$ be the corresponding walks on $(G_t)_{t\geq0}$ and $(G_t')_{t\geq0}$, respectively, both starting at $v_0$.
    Then if $(X_t)_{t\geq0}$ is transient, in the sense that it almost surely returns to $v_0$ finitely often, then so is $(Y_t)_{t\geq0}$.
\end{conjecture}
Note that because irreducible time-homogeneous random walks satisfy a 0-1 law for recurrence vs.\ transience, i.e., the probability that a random walk returns finitely often is either 0 or 1, \cref{conj:monotone_trans} immediately implies \cref{conj:dhs}, though of course \cref{conj:dhs} is now proven by \cref{thm:inc_rec}.
It is natural to wish to generalize this conjecture to the conductance setting, and we believe this can be done, but one must be cautious to preserve the uniformly bounded degree assumption in some form, e.g., uniformly bounded total conductance at each vertex.
Dembo, Huang, and Sidoravicius~\cite[Remark 1.9]{DHS2014} noted that \cref{conj:monotone_trans} does not hold without the uniformly bounded degree assumption; we provide a simpler example below.
\begin{example}\label{ex:inc_trans_unbounded}
    Let $V=\Z^3$ with $c_0\equiv1$ the uniform network corresponding to simple random walk on $\Z^3$, which is transient.
    We view $\Z^2$ as a subset of $\Z^3$ by identifying it with the plane $z=0$.
    Now let $c_t(e)=1$ for all $t\geq0$ and all edges $e$ not lying in the plane $z=0$, and for edges $e$ lying in the plane, let $c_t(e)=2^t$.
    So $(c_t)_{t\geq0}$ is monotone non-decreasing, with $c_0$ transient (and thus $c_t$ transient for all $t$ via Rayleigh monotonicity).
    Then for the random walk $(X_t)_{t\geq0}$ on $(c_t)_{t\geq0}$, the first Borel--Cantelli lemma implies that $(X_t)_{t\geq0}$ leaves the plane $z=0$ finitely often.
    Upon leaving the plane, $(X_t)_{t\geq0}$ behaves as simple random walk on $\Z^3$ until returning to the plane, and just considering the $z$-coordinate, as simple random walk on $\Z$ is recurrent, each time $(X_t)_{t\geq0}$ leaves the plane it almost surely returns to the plane.
    Thus, as $(X_t)_{t\geq0}$ leaves the plane finitely often, it eventually remains on the plane forever, in which case it behaves as simple random walk on $\Z^2$, which is recurrent.
    Thus, $(X_t)_{t\geq0}$ is recurrent in the sense that it visits any vertex in the plane $z=0$ infinitely often.

    This example demonstrates that the existence of the limit $c_\infty$ is needed for \cref{conj:inc_trans}, or equivalently \cref{thm:inc_trans}.
    And, renaming this random walk as $(Y_t)_{t\geq0}$ and letting $(X_t)_{t\geq0}$ be the simple random walk on $\Z^3$, i.e., on $c_0$, this serves as a counterexample for the conductance generalization of \cref{conj:monotone_trans} without any analog of the uniformly bounded degree assumption.
\end{example}
Because the 0-1 law does not hold in general for random walks in changing environments, the following recurrence variant of \cref{conj:monotone_trans} is a nontrivially different ``contrapositive''.
\begin{conjecture}\label{conj:monotone_rec}
    Let $(G_t)_{t\geq0}$ and $(G_t')_{t\geq0}$ be two deterministic sequences of non-decreasing graphs of uniformly bounded degrees on the same countable vertex set $V$, with $G_t \subseteq G_t'$ for all $t$.
    Let $(X_t)_{t\geq0}$ and $(Y_t)_{t\geq0}$ be the corresponding walks on $(G_t)_{t\geq0}$ and $(G_t')_{t\geq0}$, respectively, both starting at $v_0$.
    Then if $(Y_t)_{t\geq0}$ is recurrent, in the sense that it almost surely visits $v_0$ infinitely often, then so is $(X_t)_{t\geq0}$.
\end{conjecture}
Our current framework is unable to address these conjectures; in some sense, our current framework is agnostic to the growth rate of the conductances $(c_t)_{t\geq0}$, or in the $\set{0,1}$-valued setting the graphs $(G_t)_{t\geq0}$.
However, the growth rate is a critical phenomenon in \cref{conj:monotone_trans,conj:monotone_rec}; the work of Dembo, Huang, and Sidoravicius~\cite{DHS2014} shows that one can have transient $c_\infty$, yet $(c_t)_{t\geq0}$ grows sufficiently slowly that $(X_t)_{t\geq0}$ on $(c_t)_{t\geq0}$ is recurrent.

Note that \cref{conj:dec_rec} assumed $c_\infty\geq\alpha c_0$, so that all networks were within a constant factor of each other.
\cref{conj:dec_trans} does not make such an assumption.
While both \cref{conj:dec_trans,conj:dec_rec} are false, we do not know whether \cref{conj:dec_trans} can be repaired or not by assuming such a constant-factor assumption.
\begin{question}\label{question:dec_trans_alpha}
    Suppose random walk $(X_t)_{t\geq0}$ in changing environment corresponding to deterministic conductances $(c_t)_{t\geq0}$ has $(c_t)_{t\geq0}$ non-increasing with limit $c_\infty$, i.e., $c_t\downarrow c_\infty$.
    If $c_\infty\geq\alpha c_0$ for some $\alpha>0$ and random walk corresponding to $c_\infty$ is irreducible and transient, must $(X_t)_{t\geq0}$ also be transient, in the sense that almost surely the walk visits every vertex finitely often?
\end{question}
Lastly, we remark that our counterexample for \cref{conj:dec_rec} in \cref{thm:dec_rec} is not as strong as it could potentially be, and we leave open the question of whether or not a stronger counterexample can be found.
In particular, our counterexample only shows $(X_t)_{t\geq0}$ is transient with probability at least $1-\delta$, and one could hope to find a counterexample where $(X_t)_{t\geq0}$ is almost surely transient.

\section*{Acknowledgments}
We would like to thank Amir Dembo for introducing us to this problem and providing helpful discussions.
Rupert Li was partially supported by a Hertz Fellowship and a PD Soros Fellowship.
Jiyun Park was partially funded by NSF grant DMS-2348142 and a grant from the Simons Foundation International [SFI-MPS-SDF-00014916].
Claude and ChatGPT were used as auxiliary tools while preparing this manuscript.
All statements, arguments, and proofs were verified and either written or revised by the authors, who take full responsibility.

\bibliographystyle{abbrvurl}
\bibliography{ref}

\end{document}